\documentclass[10 pt]{article}
\usepackage{amssymb,amsthm}
\font\goth=eufm10
\font\bigmath=cmsy10 scaled \magstep 2
\font\subgoth=eufm8
\newcommand{\bigtimes}{\hbox{\bigmath \char'2}}
\newcommand{\gc}{\hbox{\goth c}}
\newcommand{\cf}{\hbox{\rm cf}}
\newcommand{\subgc}{\hbox{\subgoth c}}
\newcommand{\mod}{\hbox{\rm mod}}
\newcommand{\emp}{\emptyset}
\newcommand{\ben}{\mathbb N}
\newcommand{\ber}{\mathbb R}
\newcommand{\bep}{\mathbb P}
\newcommand{\beq}{\mathbb Q}
\newcommand{\bez}{\mathbb Z}

\newcommand{\supp}{\hbox{\rm supp}}
\newcommand{\nhat}[1]{\{1,2,\ldots,#1\}}
\newcommand{\ohat}[1]{\{0,1,\ldots,#1\}}
\newcommand{\pf}{{\mathcal P}_f}
\newcommand{\symdif}{\bigtriangleup}
\newtheorem{theorem}{Theorem}[section]
\newtheorem{corollary}[theorem]{Corollary}
\newtheorem{lemma}[theorem]{Lemma}
\newtheorem{question}[theorem]{Question}
\theoremstyle{definition}
\newtheorem{definition}[theorem]{Definition}

\title{Pairwise sums in colourings of the reals}
\date{}
\author{Neil Hindman
        \footnote{Department of Mathematics,
                 Howard University,
                  Washington, DC 20059, USA.\hfill\break
                  {\tt nhindman@aol.com}}
        \thanks{This author acknowledges support received from the National
                Science Foundation (USA) via Grant DMS-1160566.}
        \and
        Imre Leader
        \footnote{Department of Pure Mathematics and Mathematical Statistics,
                  Centre for Mathematical Sciences,
                  Wilberforce Road, Cambridge, CB3 0WB, UK.\hfill\break
                  {\tt leader@dpmms.cam.ac.uk}}
\and
        Dona Strauss
        \footnote{Department of Pure Mathematics,
                  University of Leeds, Leeds LS2 9J2, UK.\hfill\break
                  {\tt d.strauss@hull.ac.uk}}                 
}

\begin{document}

\maketitle

\begin{abstract} 

Suppose that we have a finite colouring of $\ber$. What sumset-type structures can we hope to
find in some colour class? One of our aims is to show that there is such a colouring
for which no uncountable set has all of its pairwise sums monochromatic. We also show that
there is such a colouring such that there is no infinite set $X$ with $X+X$ (the pairwise
sums from $X$, allowing repetition) monochromatic. These results assume CH. In the other
direction, we show that if each colour class is measurable, or each colour class is Baire,
then there is an infinite set $X$ (and even an uncountable $X$, of size the reals)  
with $X+X$ monochromatic. We also give versions for all of these results for $k$-wise sums in
place of pairwise sums. 

\end{abstract} 

\section{Introduction}

Our starting point for this paper is a question that arises from two well known statements. 
One is a standard application of Ramsey's theorem: that whenever
the natural numbers are finitely coloured (i.e.~partitioned into a finite number of classes) there is an 
infinite set $X$ with all pairwise sums
(meaning $\{x+y:\ x,y \in X,\ x \neq y\}$) monochromatic. Indeed, given such a colouring $c$, we
induce a colouring of $[ \ben ]^2$, the unordered pairs from $\ben$, by giving $\{x,y \}$ the
colour $c(x+y)$, and apply Ramsey's theorem. The other result is that there is a finite
colouring of $[ \ber ]^2$ without an uncountable monochromatic set -- one fixes a well-ordering
of the reals and then colours the pair $\{x,y \}$ according to whether the usual ordering
and this well-ordering agree or disagree on it (see \cite{EHR}).

This suggests a natural question. The `abstract' form, namely Ramsey's theorem, fails in
the uncountable case (in $\ber$), but what about the `additive' form: might it still be true
that in any finite colouring of $\ber$ there is an uncountable set with all its pairwise
sums monochromatic?

One of our main aims in this paper is to answer this question in the negative: we give a finite
colouring of $\ber$ such that there is no such uncountable set. This seems to be considerably harder
that the `abstract' question. Indeed, rather curiously, our proof relies on the 
Continuum Hypothesis (CH), unlike for the result of \cite{EHR}. We do not know whether or 
not CH is needed. Without CH, our
result asserts that there is no such set of size $\gc$ (the cardinality of the reals).

In the naturals, one cannot extend the starting result to get an infinite set $X$ with $X+X$
(meaning $\{x+y:\ x,y \in X \}$) monochromatic: this is because such a sumset automatically
contains two numbers with one roughly twice the other, and this can easily be ruled out by a 
suitable 3-colouring (see \cite{Hb}). Our other main aim is to show that 
this is also the case in the reals: there
is a finite colouring of $\ber$ with no infinite $X$ having $X+X$ monochromatic. Our proof again
uses CH. In fact, our proof goes through as long as $\gc < \aleph_\omega$, but we do not know
what happens if no such cardinal assumptions are made. The proof starts in the rationals: it turns out
that the key first step is to find not just a bad colouring of $\beq$ (meaning with no 
infinite $X$ having $X+X$ monochromatic),
but instead bad colourings for $\beq^m$, for each $m$, with the number
of colours bounded. The other main idea is then a kind of stepping-up
argument which may be of independent interest.

What happens for `nice' colourings of the reals? We show that if each colour class is measurable, or each
colour class has the property of Baire, then there does indeed exist an infinite $X$, and even
an uncountable $X$ (of size $\gc$), with $X+X$ monochromatic. It turns out that, while these
results would be reasonably straightforward just for $X$ infinite, it is quite an intricate
task to obtain uncountable $X$. We give a fairly unified treatment of the measurable and Baire
cases -  although these two cases are superficially similar, we have to work harder in the 
measurable case.

In the rest of this section we introduce (and make precise) the notation we shall be using, and
also mention some additional background and motivation. Also, our results go through for $k$-wise sums instead of
just pairwise sums, and we usually prove the results in that form -- there are sometimes some extra complications
in this general case.

We will use colouring terminology throughout this paper. A
{\it colouring\/} of a set $X$ is a function whose domain is
$X$.  A finite colouring is a colouring whose range is finite,
and for a cardinal $\kappa$, a $\kappa$-colouring is a 
colouring whose range has cardinality $\kappa$. A set $Y$
is {\it monochromatic\/} with respect to a given colouring
$\psi$ provided $\psi$ is constant on $Y$.
We write $\omega$, $\omega_1$, and $\gc$ for the first infinite
cardinal, the first uncountable cardinal, and the cardinal number of
the real line, respectively, and let $\ben=\omega\setminus\{0\}$.
We take a cardinal to be the first ordinal of a given size.  So 
$\omega$, $\omega_1$, and $\gc$ are respectively the first
infinite ordinal, the first uncountable ordinal, and the 
first ordinal with the same size as $\ber$.  Thus, in the
alternative aleph notation, $\omega=\aleph_0$ and $\omega_1=\aleph_1$.
Given a set $X$ and a cardinal $\kappa$, we let
$[X]^\kappa=\{A\subseteq X:|A|=\kappa\}$.

If one finitely colours the set $\ber$ of real numbers, it is an immediate
consequence of the Finite Sums Theorem \cite{Ha} that there is an infinite
set $X\subseteq\ber$ such that $FS(X)$ is monochromatic, where
$FS(X)=\{\sum F:F\in\pf(X)\}$
and, for any set $X$, $\pf(X)$ is the set of finite nonempty subsets of 
$X$ (and $\sum F$ denotes the sum of the elements of $F$). 

There are two natural ways to try to extend this result. The 
first is to allow repetition of terms in the sums from $X$.
It is an open question of Owings \cite{O} whether, for any
$2$-colouring of $\ben$, there is some infinite $X\subseteq\ben$ with
$X+X$ monochromatic.  On the other hand, as mentioned above, it was shown in 
\cite{Hb} that there are $3$-colourings of $\ben$ (in fact, with
one of the colour classes quite small) for which no
such $X$ exists.  

The second natural extension would be to produce an uncountable
$X$ with $FS(X)$ monochromatic.  This extension is easily seen
to be false. (See Theorem \ref{kandm} below.) So one
next asks whether one can get an uncountable $X\subseteq\ber$
with $\{x+y:\{x,y\}\in[X]^2\}$ monochromatic. 

\begin{definition}\label{defFS} Let $k\in\ben$ and let
$X\subseteq\ber$.
\begin{itemize}
\item[(a)] $kX=X+X+\ldots+X$ ($k$ times).
\item[(b)] $FS_k(X)=\{\sum F:F\in [X]^k\}$.
\end{itemize}\end{definition} 

Note that if $k>1$, then $kX\neq\{kx:x\in X\}$.
Note also that $2X=FS_2(X)\cup\{2x:x\in X\}$.

We show in Section 2 that for each $k\in\beq\setminus\{1\}$
and each cardinal $\kappa<\omega_\omega$,
there is a finite colouring of $\bigoplus_{\sigma<\kappa}\beq$  such that no
infinite $X$ has $FS_k(X)\cup\{kx:x\in X\}$ monochromatic. 
Therefore, as long as $\gc<\omega_{\omega}$, there is a finite colouring of
$\ber$ such that no infinite $X$ has $FS_k(X)\cup\{kx:x\in X\}$ monochromatic. 
In particular, no infinite $X$ has $kX$ monochromatic. 

In Section 3, we show that there is a $2$-colouring of
$\ber$ such that for any $k\in\ben\setminus\{1\}$ there
is no $X\subseteq\ber$ with $|X|=\gc$ such that
$FS_2(X)$ is monochromatic.

% The results in Section 3 use the Axiom of Choice (as do the
% results of Section 2 for $\kappa\geq \omega$).
We show in Section 4 that if a countable colouring of 
$\ber$ has all of its colour classes with the 
property of Baire or has all of its colour classes
measurable, then for each $k\in\ben\setminus\{1\}$, there is 
a set $X\subseteq\ber$ with $|X|=\gc$ such that $kX$ is monochromatic.
We mention that there has been previous work on the Ramsey theory of Baire or measurable
colourings of the reals: see for example \cite{BHW} for results on infinite sums (without
repetition) and \cite{BHL} for results on sums and products at the same time.

Finally, the reader may wonder why we consider only $k$-sums for one fixed $k$ at a 
time. So, as already referred to above, we conclude this introduction with the following 
easy result, which in fact may
be skipped as it is not needed in the sequel.

\begin{theorem}\label{kandm} Let $k,m\in\ben$ with $k<m$.
Then there is a finite colouring of $\ber$ such that there does not
exist an uncountable set $X\subseteq\ber$
with $FS_k(X)\cup
FS_m(X)$ monochromatic.
\end{theorem}

\begin{proof} Note that
it suffices to provide a finite colouring of $\ber^+=\{x\in\ber:x>0\}$.  (Then
colour $0$ at will and for $x<0$ let $x$ have the colour of $|x|$.)

We consider first the possibility that $k=1$.  In this case
colour $x\in\ber^+$ by $\lfloor \log_m(x)\rfloor\ (\mod\ 2)$.
Suppose we have an uncountable set $X\subseteq\ber$ 
with $FS_1(X)\cup
FS_m(X)$ monochromatic.
Pick $t\in\bez$ such that $$|\{x\in X:\lfloor \log_m(x)\rfloor=t\}|
\geq\omega_1\,.$$  Pick $F\in[X]^m$ such that for all $x\in F$, 
$\lfloor \log_m(x)\rfloor=t$.  Then given $x\in F$,
$m^t\leq x<m^{t+1}$ so $m^{t+1}\leq\sum F<m^{t+2}$
and thus $t+1=\lfloor \log_m(\sum F)\rfloor$, while
$t+1\not\equiv t\ (\mod\ 2)$, a contradiction.

Now assume that $k>1$. Pick $u\in\ben$ such that $k^{1+1/u}\leq m$ and let
$\alpha=k^{1/u}$.  Pick $v\in\ben$ such that $\alpha^v\leq m<\alpha^{v+1}$ and
note that $m\geq k^{1+1/u}=\alpha^{u+1}$ so $v\geq u+1$. Let $l=v+2-u$ and
color $x\in\ber^+$ by $\lfloor \log_\alpha(x)\rfloor\ (\mod\ l)$.

Suppose we have an uncountable set $X\subseteq\ber$ 
with $FS_k(X)\cup
FS_m(X)$ monochromatic.
Pick $t\in\bez$ such that $|\{x\in X:\lfloor \log_\alpha(x)\rfloor=t\}|
\geq\omega_1$.  Pick $F\in[X]^m$ such that for all $x\in F$, 
$\lfloor \log_\alpha(x)\rfloor=t$ and pick $H\in [F]^k$.

For $x\in F$, $\alpha^t\leq x<\alpha^{t+1}$ so
$\textstyle \alpha^{t+u}=k\alpha^t\leq \sum H<k\alpha^{t+1}=\alpha^{t+u+1}\,.$
Also $\alpha^{t+v}\leq m\alpha^t\leq \sum F<m\alpha^{t+1}\leq\alpha^{t+v+2}$.
Therefore $\lfloor\log_{\alpha}(\sum H)\rfloor
=t+u$ and
$\lfloor\log_{\alpha}(\sum F)\rfloor$ is either $t+v$ or $t+v+1$.
Thus $t+u\equiv t+v\ (\mod\ l)$ or $t+u\equiv t+v+1\ (\mod\ l)$ while
$0<v-u<l-1$, a contradiction.
\end{proof}

Notice that the colour classes in the proof of Theorem \ref{kandm} are all
Borel.

\section{Preventing infinite $kX$ in $\beq$ and $\ber$}

We shall need the following result, whose proof we leave as an easy exercise.

\begin{lemma}\label{nofix} Let $X$ be a finite set
and let $f:X\to X$ be a function with no fixed points.  Then there exists
$\nu:X\to\{0,1,2\}$ such that for all $x,y\in X$,
if $\nu(x)=\nu(y)$, then $f(x)\neq y$.\end{lemma}

(We mention in passing that the result also holds if $X$ is infinite: this was proved by
Katet\v ov \cite{K}.)

We let $\bep$ be the set of primes.
Throughout this section we will utilise the function $\phi$ which
we now define. 

\begin{definition}\label{defphi} Given $x\in\beq\setminus\bez$, write
$x=\frac{b}{c}$, where $b\in\bez$, $c\in \ben$,
and $b$ and $c$ are relatively prime.  Let
$F$ be the set of prime factors of $c$ and for
$p\in F$, let $n_p(x)$ be the power of $p$ in the
prime factorisation of $c$.  For $p\in F$, let
$a_p(x)$ be the unique member of $\{1,2,3,\ldots,p^{n_p(x)}-1\}$
such that 
$$\textstyle x\equiv\sum_{p\in F}\displaystyle\frac{a_p(x)}{p^{n_p(x)}}
\ (\mod\ \bez)\,.$$
For $p\in\bep\setminus F$, let $n_p(x)=0$ and $a_p(x)=0$.

If $x\in\bez$, put $a_p(x)=n_p(x)=0$ for every $p\in\bep$.

Define $\phi:\beq\to\beq$ by
$\phi(x)=\sum_{p\in\bep}\displaystyle\frac{a_p(x)}{p^{n_p(x)}}$.
\end{definition}

Note that if $p$ is one of the prime factors of the denominator 
of $x$, then $p$ does not divide $a_p(x)$.  In the proof
of Theorem \ref{thmGm}, it will be useful to note that, given 
$x$ and $y$ in $\beq$, one can compute $\phi(x+y)$ as follows.
Add the terms of $\phi(x)$ and $\phi(y)$ corresponding to the primes $p$ that divide the denominator
of $x$ or $y$ and reduce to lowest terms. If the result is $\frac{a}{p^t}$, subtract multiples of
$p^t$ from $a$ until the numerator is less than $p^t$. 

Observe that if $n_p(x)=n_p(y)$, then $n_p(x+y)\leq n_p(y)$, and equality may or may not hold.
If $n_p(x)<n_p(y)$, then $n_p(x+y)= n_p(y)$ and, by a trivial calculation,
$a_p(x+y)\equiv a_p(y)\  (\mod\ p^{n_p(y)-n_p(x)}$).

For the remainder of this section, we fix $k\in\ben\setminus\{1\}$.

\begin{definition}\label{defrp} Let $P_1=\{p\in\bep:p\hbox{ divides }k\}$ and
let $P_2=\bep\setminus P_1$.  For $p\in \bep$, let
$r_p=\min\{t\in\ben:p^t>k\}$ and let $U_p=\{a\in\bez:(a,p)=1\}$.
\end{definition}

\begin{lemma}\label{psip} Let $p\in P_2$.  Then there is a 
function $\psi_p:U_p\to\{0,1,2\}$ such that, if
$a,y\in U_p$ and $y\equiv ka\ (\mod\ p^{r_p})$,
then $\psi_p(a)\neq \psi_p(y)$.\end{lemma}

\begin{proof} Let $V=U_p\cap\nhat{p^{r_p}-1}$.
 Define $l_p:V\to V$, by $l_p(a)\equiv ka\ (\mod\ p^{r_p})$.
Given $a\in V$, since $p^{r_p}$ does not divide $k-1$ and $p$ does not divide $a$, the 
map $l_p$ has no fixed points. Let $\nu$ be as guaranteed
by Lemma \ref{nofix}, define $h_p:U_p\to V$ by $h_p(x)\equiv x\ (\mod\ p^{r_p})$,
and let $\psi_p=\nu\circ h_p$.
\end{proof}

We believe that it is possible to reduce the number of colours 
used in the following theorem at least to 18 (by adjusting the definition of
$\theta$ and combining the definitions of $f$ and $g$ in that proof).  
However, from our point of view the important 
fact is that the number of colours does not depend on $m$. Indeed, this is absolutely
critical to our proof: the kind of `stepping-up' that we will later use could not get started
if the number of colours needed grew with $m$.

\begin{theorem}\label{thmGm} Let $m\in\ben$ and let 
$G=\bigoplus_{i=0}^{m-1}\beq$.  Then there is a colouring
of $G$ in $72$ colours so that, if $\vec u\in G$ and
$X$ is an infinite subset of $G$, then there is some
$\vec x\in X$ such that $\vec x+\vec u$ and $k\vec x$ have
different colours.\end{theorem}

\begin{proof} We define $3$-colourings $f$ and $g$ of
$G$, a $2$-colouring $h$ of $G$, and a $4$-colouring
$\theta$ of $G$.

For $\vec x\in G$, let $S(\vec x\,)=\{p\in\bep:n_p(x_i)>0 
\hbox{ for some } i\in\ohat{m-1}\}$ and
if $S(\vec x\,)\neq\emp$, let $M(\vec x\,)=\max S(\vec x\,)$ and let 
$$\ell(\vec x\,)=\max\{i\in\ohat{m-1}:n_{M(\vec x\,)}(x_i)>0\}\,.$$
If $S(\vec x\,)=\emp$ (i.e., if $\vec x\in\bez^m$), let $M(\vec x\,)=f(\vec x\,)=0$.
Now assume that $S(\vec x\,)\neq\emp$, let $p=M(\vec x\,)$,
and let $j=\ell(\vec x\,)$. If $p\in P_1$, let $f(\vec x\,)=0$.
If $p\in P_2$, let $f(\vec x\,)=\psi_p\big(a_p(x_j)\big)$, 
where $\psi_p$ is as in Lemma \ref{psip}.

For $\vec x\in G$, let $N(\vec x\,)=\max\{n_p(x_i):i\in\ohat{m-1}\hbox{ and }p\in P_2\}$.
If $N(\vec x\,)=0$, let $g(\vec x\,)=0$. Now assume that $N(\vec x\,)>0$, let
$$p=\max\big\{q\in P_2:n_q(x_i)=N(\vec x\,) 
\hbox{ for some } i\in\ohat{m-1} \big\}\,,$$ let
$j=\max\{i\in\ohat{m-1}:n_p(x_i)=N(\vec x\,)\}$, and let
$g(\vec x\,)=\psi_p\big(a_p(x_j)\big)$.

For $p\in P_1$, let $m_p=\max\{t\in\ben:p^t\hbox{ divides }k\}$.
For $\vec x\in G$, let $L(\vec x\,)=\max\{\lfloor n_p(x_i)/m_p\rfloor:i\in\ohat{m-1} 
\hbox{ and }p\in P_1\}$ and define $h(\vec x\,)\in\{0,1\}$ by
$h(\vec x\,)\equiv L(\vec x\,)\ (\mod\ 2)$.

Define $\theta:G\to\{0,1,2,3\}$ by, for $\vec x\in G$,
$$\textstyle\theta(\vec x\,)\equiv\lfloor \log_{\sqrt k}(\sum_{i=0}^{m-1}|x_i|)\rfloor\ (\mod\ 4)\,.$$

Define a $72$-colouring $\gamma$ of $G$ by, for $\vec x,\vec y\in G$,
$\gamma(\vec x\,)=\gamma(\vec y\,)$ if and only if $f$, $g$, $h$, and 
$\theta$ all agree at $\vec x$ and $\vec y$.

Let $\vec u\in G$ and infinite $X\subseteq G$ be given. 
Notice that, if $M[X]$, $N[X]$, and $L[X]$ are all finite, then
$\big\{\big(\phi(x_0), \phi(x_1),\ldots,\phi(x_{m-1})\big):\vec x\in X\}$
is finite.  Thus one of the following four cases must hold.
\begin{itemize}
\item[(I)] $\big\{\big(\phi(x_0), \phi(x_1),\ldots,\phi(x_{m-1})\big):\vec x\in X\}$
is finite;
\item[(II)] $M[X]$ is infinite;
\item[(III)] $N[X]$ is infinite; or
\item[(IV)] $L[X]$ is infinite.
\end{itemize}  
We will show that in each of these cases there is some $\vec x\in X$ such
that $\gamma(\vec u+\vec x\,)\neq \gamma(k\vec x\,)$.

Case (I).  We may assume that the mapping $\vec x\mapsto \big(\phi(x_0), \phi(x_1),\ldots,\phi(x_{m-1})\big)$ is constant on $X$.  
Then, for any $\vec x,\vec y\in X$,
$\vec y\in \vec x+\bez^m$.  It follows that, if $\sigma:G\to \ber$ is defined by
$\sigma(\vec x\,)=\sum_{i=0}^{m-1}|x_i|$, then $\sigma[X]$ is unbounded. 
For any $\vec u\in G$, we have $\sigma(\vec x\,)-\sigma(\vec u\,)\leq \sigma(\vec u+\vec x\,)
\leq\sigma(\vec x\,)+\sigma(\vec u\,)$. So, given $\vec u\in G$, we can choose $\vec x\in X$ such that
$$\log_{\sqrt{k}}\big(\sigma(\vec x\,)\big)-1<\log_{\sqrt{k}}\big(\sigma(\vec u+\vec x\,)\big)
<\log_{\sqrt{k}}\big(\sigma(\vec x\,)\big)+1\,.$$ 
 Since $\log_{\sqrt{k}}(k\vec x\,)=\log_{\sqrt{k}}(\vec x\,)+2$, it follows that 
$\theta(\vec u+\vec x\,)\neq \theta(\vec x\,)$.

Case (II). Pick $\vec x\in X$ such that $p=M(\vec x\,)>M(\vec u\,)$ and
$p>k$.  Let $j=\ell(\vec x)$.
Then $p=M(\vec u+\vec x\,)$, $j=\ell(\vec u+\vec x\,)$,
and $a_p(u_{j}+x_{j})=a_p(x_{j})$. 
Also, because $p\in P_2$, $p= M(k\vec x\,)$ and $j=\ell(k\vec x\,)$. Since
$a_p(kx_{j})\equiv ka_p(x_{j})\ (\mod\ p^{n_p(x_{j})})$, it follows that
$a_p(kx_{j})\equiv ka_p(x_{j})\ (\mod\  p^{r_p})$ because $r_p=1$.
So, by Lemma \ref{psip}, $f(\vec u+\vec x\,)\neq f(k\vec x)$.

Case (III). We can choose  $\vec x\in X$ such that $N(\vec x\,)>N(\vec u\,)+r_2$, and so
$N(\vec x\,)>N(\vec u\,)+r_p$ for every $p\in\bep$.
Let $$p=\max\big\{q\in P_2:n_q(x_i)=N(\vec x\,) \hbox{ for some } i\in\ohat{m-1} \big\}$$
and let $j=\max\{i\in\ohat{m-1}:n_p(x_i)=N(\vec x\,)\}$.  Now, if
$q\in P_2$, $q>p$ and $i\in\ohat{m-1}$, then
$n_q(x_i)<n_p(x_j)$ and $n_q(u_i)\leq N(\vec u\,)<n_p(x_j)$ so
$n_q(u_i+x_i)\leq \max\{n_q(u_i),n_q(x_i)\}<n_p(x_j)$.
Likewise, if $i\in\ohat{m-1}$ and $i>j$,
then $n_p(x_i)<n_p(x_j)$ and $n_p(u_i)<n_p(x_j)$.  Thus,
$N(\vec u+\vec x\,)=N(\vec x\,)$, 
$$p=\max\big\{q\in P_2:n_q(u_i+x_i)=N(\vec u+\vec x\,) \hbox{ for some } i\in\ohat{m-1} \big\}\,,$$
and $j=\max\{i\in\ohat{m-1}:n_p(u_i+x_i)=N(\vec u+\vec x\,)\}$.

Note that, for every $q\in P_2$ and every $x\in \beq$, $n_q(kx)=n_q(x)$. 
Thus $N(k\vec x\,)=N(\vec x)$, $p=\max\big\{q\in P_2:(\exists i\in\ohat{m-1})\big(n_q(kx_i)=N(k\vec x\,)\big)\big\}\,,$
and $j=\max\{i\in\ohat{m-1}:n_p(kx_i)=N(k\vec x\,)\}$

We have observed that $a_p(u_j+x_j)\equiv a_p(x_j)\ (\mod\ p^{n_p(x_j)-n_p(u_j)}$),
and so $a_p(u_j+x_j)\equiv a_p(x_j)\ (\mod\ p^{r_p})$ because $n_p(x_j)-n_p(u_j)>r_p$.
Also, $a_p(kx_j)\equiv ka_p(x_j)\ (\mod\ p^{n_p})$ and so 
$a_p(kx_j)\equiv ka_p(x_j)\ (\mod\ p^{r_p}$). It follows from Lemma \ref{psip} 
that $g(\vec u+\vec x)\neq g(k\vec x)$.

Case (IV). Using the pigeonhole principle, we may presume that we have $p\in P_1$ and
$j\in\ohat{m-1}$ such that for all $\vec x\in X$, 
$$p=\max\big\{q\in P_1:\lfloor n_q(x_i)/m_q\rfloor=L(\vec x\,) \hbox{ for some } 
i\in\ohat{m-1} \big\}$$
and $j=\max\{i\in\ohat{m-1}:\lfloor n_p(x_i)/m_p\rfloor=L(\vec x\,)\}$.
Pick $\vec x\in X$ such that $L(\vec x\,)>L(\vec u)$ and let
$l=\lfloor n_p(x_j)/m_p\rfloor=L(\vec x\,)$.  We show first
that $L(\vec u+\vec x)=l$.  Since $n_p(x_j)>n_p(u_j)$,
we have $n_p(u_j+x_j)=n_p(x_j)$ so $\lfloor n_p(u_j+x_j)/m_p\rfloor=l$
so $l\leq L(\vec u+\vec x)$.  Suppose we have $q\in P_1$
and $i\in \ohat{m-1}$ such that $l<\lfloor n_q(u_i+x_i)/m_q\rfloor$.
Since $n_q(u_i+x_i)\leq\max\{n_q(x_i),n_q(u_i)\}$,
we have $l<\lfloor n_q(u_i)/m_q\rfloor\leq L(\vec u\,)<L(\vec x\,)$, 
a contradiction. 

Now we claim that $L(k\vec x)=l-1$.  For any $q\in P_1$ and $i\in\ohat{m-1}$,
$n_q(kx_i)=\max\{n_q(x_i)-m_q,0\}$ so
$\lfloor n_q(kx_i)/m_q\rfloor\leq \lfloor n_q(x_i)/m_q\rfloor -1\leq l-1$.
Also since $n_p(x_j)\geq m_p$, $\lfloor n_p(kx_j)/m_p\rfloor = \lfloor n_p(x_j)/m_p\rfloor -1= l-1$.
Thus $h(\vec u+\vec x\,)\neq h(k\vec x\,)$.
\end{proof}

Notice that the case $m=1$ of Theorem \ref{thmGm} establishes
that $\beq$ can be finitely coloured so that there is 
no infinite $X$ with $kX$ monochromatic.

\begin{definition}\label{defGlambda} For each cardinal $\kappa>0$, $G(\kappa)=\bigoplus_{\sigma<\kappa}\beq$.
\end{definition}

\begin{lemma}\label{Gkappa} Let $\kappa$ be an infinite cardinal. Assume that there exists $n\in \ben$ such that, for every cardinal
$\lambda$ with $1\leq\lambda<\kappa$, there is an $n$-colouring of 
$G(\lambda)$ with the property that there is no infinite subset $X$ of $G(\lambda)$ for which
$FS_k\langle X\rangle \cup \{k\vec x:\vec x\in X\}$ is monochromatic. 
Then there is a $2n$-colouring of $G(\kappa)$ such that there is no infinite subset $X$ of $G(\kappa)$ for which
$FS_k\langle X\rangle \cup \{k\vec x:\vec x\in X\}$ is monochromatic. 
\end{lemma}

\begin{proof} Let $\xi$ be a $2$-colouring of $\beq\setminus\{0\}$, such that
$t$ and $kt$ have different colours for every $t\in\beq\setminus\{0\}$.  For
$\vec x\in G(\kappa)\setminus\{\vec 0\,\}$, let $\mu(\vec x\,)=\max\{\sigma<\kappa:x_\sigma\neq 0\}$
and let $\eta(\vec x)=x_{\mu(\vec x\,)}$.  For each $\alpha<\kappa$, let
$H_\alpha=\{\vec 0\,\}\cup\{\vec x\in G(\kappa)\setminus\{\vec 0\,\}:\mu(\vec x\,)\leq\alpha\}$.
Then $H_\alpha$ is isomorphic to $G(|\alpha+1|)$ so pick an $n$-colouring $\gamma_\alpha$ of
$H_\alpha$ such that there is no infinite subset $X$ of $H_\alpha$ for which
$FS_k\langle X\rangle \cup \{k\vec x:\vec x\in X\}$ is monochromatic. 
Define a $2n$-colouring $\tau$ of $G(\kappa)$ by, for 
$\vec x\in G(\kappa)\setminus\{\vec 0\,\}$, $\tau(\vec x\,)=\big(\gamma_{\mu(\vec x\,)}(\vec x\,),\xi\big(\eta(\vec x)\big)\big)$,
assigning $\tau(\vec 0\,)$ arbitrarily.

Suppose we have an infinite subset $X$ of $G(\kappa)$ for which
$FS_k\langle X\rangle \cup \{k\vec x:\vec x\in X\}$ is monochromatic with respect to $\tau$.
We may assume $\vec 0\notin X$.
Suppose first that $|\mu[X]|\geq k$ and pick $\vec x_1,\vec x_2,\ldots,\vec x_k$ in $X$
such that $\mu(\vec x_1)<\mu(\vec x_2)<\ldots<\mu(\vec x_k)$.
Then $\xi\big(\eta(\vec x_1+\ldots+\vec x_k)\big)=\xi\big(\eta(\vec x_k)\big)
\neq \xi\big(k\eta(\vec x_k)\big)=\xi\big(\eta(k\vec x_k)\big)$, a contradiction.
Thus, by the pigeonhole principle, we may assume we have $\alpha<\kappa$ such that
$\mu[X]=\{\alpha\}$.  Then $X$ is an infinite subset of $H_\alpha$ such that
$FS_k\langle X\rangle \cup \{k\vec x:\vec x\in X\}$ is monochromatic with respect to
$\gamma_\alpha$, a contradiction.\end{proof}

Since $\ber$ is isomorphic to $G(\gc)$, the following theorem provides a CH proof
that there is a finite colouring of $\ber$ with the property that there is no 
infinite subset $X$ of $\ber$ for which $FS_k\langle X\rangle\cup \{kx:x\in X\}$ is monochromatic.
In fact, since the cofinality of $\gc$ is uncountable, the smallest possible value
for $\gc$ for which this assertion might fail is $\omega_{\omega+1}$.

\begin{theorem}\label{Gomegan} Let $n<\omega$.  
Then there is a colouring of $G(\omega_n)$ by $2^{4+n}\cdot 3^2$
colours such that there is no infinite subset $X$ of $G(\omega_1)$ for which 
$FS_k\langle X\rangle\cup \{k\vec x:\vec x\in X\}$ is monochromatic. 
In particular, there is no infinite subset $X$ of $G(\omega_n)$ for which $kX$ is monochromatic. \end{theorem}

\begin{proof} It follows from Theorem \ref{thmGm} and Lemma \ref{Gkappa} 
that there is a colouring of $G(\omega)$ by $2^4\cdot 3^2$ colours
such that there is no infinite subset $X$ of $G(\omega)$ for which $FS_k\langle X\rangle\cup \{kx:x\in X\}$ 
is monochromatic. The conclusion then follows by induction using Lemma \ref{Gkappa}.\end{proof}

\begin{question} Can one show in ZFC, without extra set theoretic assumptions, that there is 
a finite colouring of $\ber$ such that there is no infinite subset $X$ of $\ber$ for which
$FS_k\langle X\rangle\cup \{kx:x\in X\}$ is monchromatic? \end{question}

\section{Preventing uncountable $FS_k(X)$ in $\ber$}

It has been known at least since the publication of \cite{EHR} that
there is a two-colouring of $[\ber]^2$ such that no uncountable
$X$ has $[X]^2$ monochromatic.  Indeed, let $W$ be a given well-ordering
of $\ber$ and colour the pair $\{x,y\}$ colour $1$ if $x<y$ and $x\,W\,y$
and colour $2$ if $x<y$ and $y\,W\,x$.  No uncountable subset
of $\ber$ is either well-ordered or reverse well-ordered by $<$, because
between each element and its successor, there must be a rational. Thus one has
that there is no uncountable set $X$ with $[X]^2$ monochromatic. 

This makes the statement of Theorem \ref{main} very believable. However, it seems to be
significantly harder to show -- as witnessed by the fact that we are only able to show that there 
is no $X$ of size $\gc$.

We omit the routine proof of the following lemma.

\begin{lemma}\label{lemomega1}  Let $Y\subseteq \ber$ such that $|Y|=\omega_1$
and let $$A=\{x\in Y:|\{y\in Y:x<y\}|<\omega_1\}\,.$$
Then $|A|<\omega_1$.\end{lemma}

\begin{theorem}\label{main} There is a $2$-colouring of $\ber$ such that,
given any $k\in\ben\setminus\{1\}$, 
there does not exist a set $X\subseteq \ber$ with $|X|=\gc$
such that $FS_k(X)$ is monochromatic.
\end{theorem}

\begin{proof} Fix a Hamel basis $\langle e_i\rangle_{i\in\ber}$ for
$\ber$ over $\beq$ and fix a well-ordering $W$ of $\ber$ of order
type $\gc$.  For each $x\in\ber$, let $S(x)\subseteq\ber$ and
$\alpha(x):S(x)\to\beq\setminus\{0\}$ such that
$x=\sum_{i\in S(x)}\alpha(x)(i)\cdot e_i$. (Then $S(0)=\emp$ and 
for all $x\in\ber\setminus\{0\}$, $S(x)$ is a finite nonempty 
subset of $\ber$, called the {\it support\/} of $x$.)

We now define a colouring $\psi:\ber\to\{0,1\}$. Let $\psi(0)=0$. 
Now let $x\in\ber\setminus\{0\}$ be given.  Let 
$m=|S(x)|$ and let $i_1<i_2<\ldots<i_m$ such that
$S(x)=\{i_1,i_2,\ldots,i_m\}$.  Pick $t\in\nhat{m}$ such that
for all $s\in\nhat{m}\setminus\{t\}$,
$i_s\,W\,i_t$. Let $\psi(x)\equiv t\ (\mod\ 2)$.

Let $k\in\ben\setminus\{1\}$ be given.  We shall write the
proof assuming that $k=2$, showing at the conclusion how to
modify the proof for larger values of $k$.

Suppose we have a set $X\in[\ber]^{\subgc}$ with 
$FS_2(X)$ monochromatic with respect to $\psi$. We shall repeatedly observe that
there are many elements of $X$ with a particular property
and assume then (by throwing the others away) that all elements
of $X$ have that property.  (At one stage in the proof, 
``many'' changes from $\gc$ to $\omega_1$.)

Since $\cf(\gc)>\omega$, we may presume that we have some 
$m\in\ben$ such that for all $x\in X$, $|S(x)|=m$.
For each $x\in X$, let $i(x,1)<i(x,2)<\ldots<i(x,m)$
such that $S(x)=\{i(x,1),i(x,2),\ldots,i(x,m)\}$.

Choose $I:[\ber]^m\to\beq^m$ such that if 
$i_1<i_2<\ldots<i_m$ and $I(\{i_1,i_2,\ldots,i_m\})\break
=(s_1,s_2,\ldots,s_m)$, then 
$s_1<i_1<s_2<i_2<\ldots<i_{m-1}<s_m<i_m$.  Again using the
fact that $\cf(\gc)>\omega$, we may presume that
we have $(s_1,s_2,\ldots,s_m)\in\beq^m$ such that
for all $x\in X$, we have $I\big(S(x)\big)=(s_1,s_2,\ldots,s_m)$.
As a consequence if $x,y\in X$ and $t\in\nhat{m-1}$, 
then $i(x,t)<s_{t+1}<i(y,t+1)$.

For each $x\in X$, let $\delta(x)\in\nhat{m}$ such that
$$\hbox{for all }j\in\nhat{m}\setminus\{\delta(x)\}\,,\,
i(x,j)\,W\,i\big(x,\delta(x)\big)\,.$$  We may presume
that we have $l\in\nhat{m}$ such that for all $x\in X$, 
$\delta(x)=l$.

Now $\{i(x,l):x\in X\}$ is cofinal in $W$ since,
given $t\in\ber$, $|\{x\in\ber:S(x)\subseteq\{s\in\ber:s\,W\,t\}|<\gc$.
(Recall that $W$ has order type $\gc$.)  Therefore
$|\{i(x,l):x\in X\}|\geq \cf(\gc)\geq\omega_1$.
Thus we may assume that we have $Y\in [X]^{\omega_1}$
such that for all $x,y\in Y$, if $x\neq y$, then 
$i(x,l)\neq i(y,l)$.

We may assume that 
for each $j\in\nhat{m}$, either
\begin{itemize}
\item[(1)] for all $x,y\in Y$, $i(x,j)=i(y,j)$ or
\item[(2)] for all $x\neq y$ in $Y$, $i(x,j)\neq i(y,j)$.
\end{itemize}

\noindent
Let $M=\{j\in\nhat{m}:$ for all $x,y\in Y$, $i(x,j)=i(y,j)\}$
and note that $l\notin M$.  
We may presume that for each $j\in M$, either
\begin{itemize}
\item[(1)] for all $x\in Y$, $\alpha\big(x,j)\big)>0$ or
\item[(2)] for all $x\in Y$, $\alpha\big(x,j)\big)<0$.
\end{itemize} 

\noindent Consequently, if $j\in M$, $u$ is the fixed value of 
$i(x,j)$, and $F$ is a finite subset
of $Y$, then $u\in S(\sum F)$.

Let $B=\{i(x,l):x\in Y\}$.
We claim that either $<$ and $W$ agree on $B$ or
$>$ and $W$ agree on $B$. This contradiction will complete the
proof.

Suppose instead we have $x$, $y$, $w$, and $z$ in $Y$ such that
\begin{itemize}
\item[(1)] $i(x,l)<i(y,l)$ and $i(x,l)\,W\,i(y,l)$ and
\item[(2)] $i(w,l)<i(z,l)$ and $i(z,l)\,W\,i(w,l)$.
\end{itemize}

Then $x+y$ and $w+z$
are members of $FS_2(Y)$.

For all $t\in S(x+y)\setminus\{i(y,l)\}$, $t\, W\,i(y,l)$ 
and $\{t\in S(x+y):t<i(y,l)\}=
\{i(x,j):j\leq l\hbox{ and }j\notin M\}\cup\{i(y,j):j< l\}$.

For all $t\in S(w+z)\setminus\{i(w,l)\}$, $t\, W\,i(w,l)$ 
and $\{t\in S(w+z):t<i(w,l)\}=
\{i(z,j):j< l\hbox{ and }j\notin M\}\cup\{i(w,j):j< l\}$.

Thus $|\{t\in S(x+y):t<i(y,l)\}|=
|\{t\in S(w+z):t<i(w,l)\}|+1$ so
that $\psi(x+y)\neq
\psi(w+z)$, a contradiction. 

This concludes the proof in the case $k=2$. Now assume that $k>2$ and
suppose we have a set $X\in[\ber]^{\subgc}$ with
$FS_k(X)$ monochromatic
with respect to $\psi$.  The proof 
proceeds verbatim through the sentence which
concludes ``then $u\in S(\sum F)$."
At this stage
let $A=\{x\in Y:|\{y\in Y:i(x,l)<i(y,l)\}|<\omega_1\}$.
By Lemma \ref{lemomega1}, $|A|<\omega_1$ so, replacing $Y$ by $Y\setminus A$
we may presume that
$$\hbox{for all }x\in Y\,,\,|\{y\in Y:i(x,l)<i(y,l)\}|=\omega_1\,.$$
Fix $x_1,x_2,\ldots,x_{k-2}\in Y$ 
such that $i(x_1,l)<i(x_2,l)<\ldots<i(x_{k-2},l)$.
We may presume that for all $y\in Y\setminus\{x_1,x_2,\ldots,x_{k-2}\}$, 
$i(y,l)>i(x_{k-2},l)$.

Then let $B=\big\{i(y,l):y\in Y\setminus\{x_1,x_2,\ldots,x_{k-2}\}\big\}$.
One shows that either $<$ and $W$ agree on $B$ or
$>$ and $W$ agree on $B$.

One picks $x$, $y$, $w$, and $z$ in $Y\setminus\{x_1,x_2,\ldots,x_{k-2}\}$ 
as before.

Let $b=x_1+x_2+\ldots+x_{k-2}$.
Then $b+x+y$ and $b+w+z$
are members of $FS_k(Y)\subseteq FS_k(X)$.

Let $D=\bigcup_{s=1}^{k-2}\{i(x_s,j):
j\leq l\hbox{ and }j\notin M\}$.

For all $t\in S(b+x+y)\setminus\{i(y,l)\}$, $t\, W\,i(y,l)$ 
and 
$$\begin{array}{rl}
\{t\in S(b+x+y):t<i(y,l)\}=
&\hskip -5 pt D\cup\{i(x,j):j\leq l\hbox{ and }j\notin M\}\cup{}\\
&\hskip -5 pt \{i(y,j):j< l\}\,.\end{array}$$

For all $t\in S(b+w+z)\setminus\{i(w,l)\}$, $t\, W\,i(w,l)$ 
and 
$$\begin{array}{rl}\{t\in S(b+w+z):t<i(w,l)\}=
&\hskip -5 pt D\cup\{i(z,j):j< l\hbox{ and }j\notin M\}\cup{}\\
&\hskip -5 pt \{i(w,j):j< l\}\,.\end{array}$$
One then reaches the same contradiction as before.
\end{proof}

Note that, if $\gc>\omega_1$, then there is a set $X\in[\ber]^{\omega_1}$
such that for each $k\in\ben\setminus\{1\}$, $FS_k(X)$ is
monochromatic with respect to the colouring $\psi$ of Theorem \ref{main}.
(To see this, pick $j\in\ber$ such that $|\{i\in\ber:i\,W\,j\}|=\omega_1$
and let $X=\{e_i+e_j:i\in\ber$ and $i\,W\,j\}$.)

\begin{question}\label{noch} Can one show in ZFC, without extra set theoretic
assumptions, that there is a finite colouring of $\ber$ such that there
is no uncountable set $X\subseteq\ber$ with $FS_2(X)$
monochromatic?\end{question}

\section{Baire and measurable colourings of $\ber$}

The subsets of $\ber$ with the property of Baire are the members
of the smallest sigma algebra containing the open sets and the
meagre sets. The set $A\subseteq \ber$ has the property of
Baire if and only if there exist an open set $U$ and a meagre
set $M$ such that $A=U\symdif M$.  By ``measurable'' we mean
``Lebesgue measurable". 

% The reader may have noticed that Theorem \ref{main} uses the axiom of choice. The theorems in 
% this section show that Theorem \ref{main} could not have been proved within ZF,
% since it is consistent with the axioms of ZF that every subset of $\ber$ is Lebesgue
% measurable. It is also consistent with the axioms of ZF that every subset of $\ber$ has the
% Baire property.

We remind the reader that if $k\in\ben$
and $A\subseteq \ber$, then $kA=A+A+\ldots+A$\break ($k$ times).  On
the other hand, by ${1\over k}A$ we mean $\{{1\over k}x:x\in A\}$ and by
$-x+A$ we mean $\{-x+y:y\in A\}$.

We note that the property with which we are concerned in this section
is translation invariant.

\begin{lemma}\label{lemtran} Let $A\subseteq\ber$, let $x\in\ber$, let
$k\in\ben$, and let $\kappa$ be a cardinal.  Then there exists $X\in[\ber]^\kappa$
such that $kX\subseteq A$ if and only if there exists $X\in[\ber]^\kappa$
such that $kX\subseteq -x+A$.\end{lemma}

\begin{proof} If $kX\subseteq A$ and $Y=-{x\over k}+X$, then
$kY\subseteq -x+A$.\end{proof}

The following theorem is a corollary to Theorem \ref{kHgc}, but its
proof is simple and self contained, so we present it separately.

\begin{theorem}\label{bairekXomegaone}  Let $Z$ be a nonmeagre subset of $\ber$ with the property of
Baire and let $k\in\ben\setminus\{1\}$.  
Then there is an uncountable 
subset $H$ of $\ber$ such that $kH\subseteq Z$. \end{theorem}

\begin{proof} Pick a nonempty open set $U$ and a meagre set $M$ such that 
$Z=U\symdif M$. By Lemma \ref{lemtran} we may presume that we have
some $\delta>0$ such that $(0,k\delta)\subseteq U$.
Pick $y_0\in(0,\delta)\setminus \frac{1}{k}M$.

Let $0<\sigma<\omega_1$ and assume we have chosen $\langle y_\tau\rangle_{\tau<\sigma}$
such that if $\mu<\tau<\sigma$, then $y_\mu\neq y_\tau$ and $y_\tau<\delta$. Assume also that for
all $\eta<\sigma$,
if $\tau(1)\leq \tau(2)\leq\ldots\leq\tau(k)\leq \eta$, then $\sum_{i=1}^k y_{\tau(i)}\in Z$.
Let $Y_\sigma=\{y_\tau:\tau<\sigma\}$ and note that $kY_\sigma\subseteq Z$.

We claim that $(0,\delta)\subseteq {1\over k}U\cap\bigcap_{r=1}^{k-1}\bigcap\{{1\over r}(-a+U):a\in(k-r)Y_\sigma\}$.
One has immediately that $(0,k\delta)\subseteq U$, so let $r\in\nhat{k-1}$ and $a\in (k-r)Y_\sigma$ be given.
Now $a<(k-r)\delta$ so $a+r\delta<k\delta$
and consequently $(0,r\delta)\subseteq (-a+U)$.
Let $B={1\over k}Z\cap\bigcap_{r=1}^{k-1}\bigcap\{{1\over r}(-a+Z):a\in(k-r)Y_\sigma\}$.
We claim that $(0,\delta)\setminus B$ is meagre.
To see this, it suffices to show that
$$\textstyle (0,\delta)\setminus B\subseteq {1\over k}M\cup\bigcup_{r=1}^{k-1}\bigcup\{{1\over r}(-a+M):a\in(k-r)Y_\sigma\}\,.$$
So let $y\in(0,\delta)$ and assume that $y\notin B$.
If $y\notin {1\over k}Z$, then since $y\in {1\over k}U$, we have
$y\in {1\over k}M$.  So assume we have $r\in\nhat{k-1}$ and $a\in (k-r)Y_\sigma$ such that
$y\notin {1\over r}(-a+Z)$. Then $y\in {1\over r}(-a+U)$ so $y\in {1\over r}(-a+M)$.

Pick $y_{\sigma}\in \big((0,\delta)\cap B\big)\setminus\{y_\tau:\tau<\sigma\}$.  
To verify the induction hypothesis, let $\tau(1)\leq \tau(2)\leq\ldots\leq \tau(k)\leq \sigma$.
We shall show that $\sum_{i=1}^k y_{\tau(i)}\in Z$.  If $\tau(k)<\sigma$, the
conclusion holds by the induction hypothesis, so assume that $\tau(k)=\sigma$.
Pick $s\in\nhat{k}$ such that $\tau(s)=\sigma$ and, if $s>1$, then $\tau(s-1)<\sigma$.
If $s=1$, then $\sum_{i=1}^k y_{\tau(i)}=ky_{\sigma}\in Z$. So assume
$s>1$ and let $r=k-s+1$. Then
$a=\sum_{i=1}^{s-1}y_{\tau(i)}\in (k-r)Y_\sigma$ so
$\sum_{i=1}^k y_{\tau(i)}=a+ry_{\sigma}\in Z$.

The construction being complete, let $H=\{y_\sigma:\sigma<\omega_1\}$.
\end{proof}

\begin{corollary}\label{bairecolor} Let $k\in\ben\setminus\{1\}$ and let
$\ber$ be countably coloured so that each colour class has the property
of Baire. Then there exists an uncountable set $H$ such that $kH$ is monochromatic.
\end{corollary}

\begin{proof} One of the colour classes must be nonmeagre. \end{proof}

The rest of this paper is devoted to the proof of the following theorem
and corollary.  We denote the Lebesgue measure of a subset $Z$ of $\ber$ by 
$\lambda(Z)$.

\begin{theorem}\label{kHgc} Let $k\in\ben\setminus\{1\}$ and let $Z\subseteq (0,1)$ such that either
$Z$ is nonmeagre with the property of Baire or $Z$ is measurable
with $\lambda(Z)>0$.  Then there exists $H\subseteq \ber$ such that $|H|=\gc$ and $kH\subseteq Z$.
\end{theorem}

\begin{corollary}\label{corkHgc} Let $k\in\ben\setminus\{1\}$ and let
$\ber$ be countably coloured so that each colour class has the property
of Baire or each colour class is measurable. Then there exists $H\subseteq\ber$  such that 
$|H|=\gc$ and $kH$ is monochromatic.\end{corollary}

To prove Theorem \ref{kHgc}, we let $k\in\ben\setminus\{1\}$ be given,
let $m=k+2$, and fix some notation which will be used throughout the proof.

\begin{definition} \begin{itemize}
\item[(a)] Let $A=\bigtimes_{n=1}^\infty\ohat{m-1}$, let
$\ohat{m-1}$ be discrete, give $A$ the product topology,
and let $\mu$ denote the product probability measure on
$A$ determined by assigning measure $\frac{1}{m}$ to each
point of $\ohat{m-1}$.
\item[(b)] Let $C$ be the set of points in $(0,1)$ that have
a terminating base $m$ expansion and let $W=(0,1)\setminus C$.
\item[(c)] Let $B=\{\alpha\in A:(\forall n\in\ben)(\exists r>n)(\exists s>n)
(\alpha(r)\neq 0\hbox{ and }\alpha(s)\neq m-1)\}$.
\item[(d)] Define $\psi:B\to W$ by $\psi(\alpha)=\sum_{n=1}^\infty\frac{\alpha(n)}{m^n}$.
\item[(e)] For $\alpha\in A$, $\supp(\alpha)=\{n\in\ben:\alpha(n)>0\}$.
\item[(f)] For $F\subseteq\ben$, $Q_F=\{\alpha\in A:\supp(\alpha)\subseteq F\}$.
\end{itemize}\end{definition}

The assertions in the following lemma can all be established
using standard techniques.

\begin{lemma}\label{lembairemeas} Let $D\subseteq(0,1)$ and let $G\subseteq B$. 
\begin{itemize}
\item[(1)] The function $\psi$ is a homeomorphism from $B$ onto $W$.
\item[(2)] If $D$ has the property of Baire in $(0,1)$, then
$D\cap W$ has the property of Baire in $W$.
\item[(3)] If $D$ is measurable in $(0,1)$, then $D\cap W$ is 
measurable in $W$.
\item[(4)] If $G$ has the property of Baire in $B$, then
$G$ has the property of Baire in $A$.
\item[(5)] If $G$ is measurable (with respect to $\mu$) in $B$, then $G$ is 
measurable in $A$.
\item[(6)] If $D$ is a Lebesgue measurable subset of $W$,
then $\psi^{-1}[D]$ is measurable in $B$ with respect to $\mu$
and $\lambda(D)=\mu(\psi^{-1}[D])$.
\end{itemize}\end{lemma}

The following lemma will be used to proof both parts of 
Theorem \ref{kHgc}.

\begin{lemma}\label{Ebig} Let $P\subseteq B$,
$\alpha\in P$, and $X$ an infinite subset of
$\ben$ such that 
\begin{itemize}
\item[(a)] for all $n\in X$, $\alpha(n)=0$ and
\item[(b)] for all $\beta\in A$, if $(\forall n\in\ben\setminus X)\big(\beta(n)=\alpha(n)\big)$,
then $\beta \in P$.
\end{itemize}
Then there exists $H\subseteq W$ such that $|H|=\gc$ and $kH\subseteq\psi[P]$.
\end{lemma}

\begin{proof} Let $E$ be the set of all $\delta\in A$ such that
\begin{itemize}
\item[(1)] for all $n\in X$, $\delta(n)\in\{0,1\}$,
\item[(2)] $\{n\in X:\delta(n)=1\}$ is infinite, and
\item[(3)] for all $n\in \ben\setminus X$, $\delta(n)=0$.
\end{itemize}

Then $|E|=\gc$ and $E\subseteq B$.  Let $H=\frac{\psi(\alpha)}{k}+\psi[E]$.
The fact that $H\subseteq W$ will follow from the fact that $kH\subseteq\psi[P]$.
To see that $kH\subseteq\psi[P]$, let $b_1,b_2,\ldots,b_k\in H$ and pick
$\delta_1,\delta_2,\ldots,\delta_k\in E$ such that for each $t\in\nhat{k}$,
$b_t=\frac{\psi(\alpha)}{k}+\psi(\delta_t)$.  Then
$\sum_{t=1}^k b_t=\psi(\alpha)+\sum_{t=1}^k\psi(\delta_t)$.  We
claim that $\psi(\alpha)+\sum_{t=1}^k\psi(\delta_t)=\psi(\alpha+\sum_{t=1}^k\delta_t)$.

Given $n\in\ben$ and $t\in\nhat{k}$, if $n\in X$, then
$\delta_t(n)\in\{0,1\}$ while if $n\in\ben\setminus X$,
then $\delta_t(n)=0$. Therefore 
$(\alpha+\sum_{t=1}^k\delta_t)(n)=\alpha(n)+\sum_{t=1}^k\delta_t(n)$
so that $\psi(\alpha+\sum_{t=1}^k\delta_t)=\psi(\alpha)+\sum_{t=1}^k\psi(\delta_t)$
as claimed.

It thus suffices to show that $\alpha+\sum_{t=1}^k\delta_t\in P$.
Now given $n\in\ben\setminus X$, $(\alpha+\sum_{t=1}^k\delta_t)(n)=\alpha(n)$,
so by (b), $\alpha+\sum_{t=1}^k\delta_t\in P$.\end{proof}

We remark that the argument in the next to last paragraph of
the proof of Lemma \ref{Ebig} is the reason for taking $m=k+2$.
If we had $m=k+1$, $\alpha(n)=m-1$ for $n\in \ben\setminus X$,
and $\delta_t(n)=1$ for $n\in X$, then one would not
have $\alpha+\sum_{t=1}^k\delta_t\in B$.

To establish Theorem \ref{kHgc} in the Baire case, we will use the
following result of Moran and Strauss.

\begin{theorem}\label{thmMS} Let $f:A\to\omega$ such that for
each $n\in\omega$, $f^{-1}[\{n\}]$ has the property of Baire.
Then there exists $U\subseteq A$ such that $A\setminus U$ is
meagre and for every $\alpha\in U$, there is a finite
set $F\subseteq \ben$ such that whenever $G$ is a finite subset
of $\ben\setminus F$ and $Y$ is an infinite subset of $\ben$,
there is an infinite $X\subseteq Y$ such that
$$(\forall\beta\in A)\big((\forall n\in\ben\setminus(G\cup X))
(\beta(n)=\alpha(n))\Rightarrow f(\beta)=f(\alpha)\big)\,.$$
\end{theorem}

\begin{proof} This is a special case of \cite[Theorem 2]{MS}.\end{proof}

\begin{lemma}\label{lembaire} Let $Z$ be a nonmeagre subset
of $(0,1)$ with the property of Baire and let 
$P=\psi^{-1}[Z\cap W]$. Then there exist $\alpha\in P$ and an infinite
subset $X$ of $\ben$ such that 
\begin{itemize}
\item[(a)] for all $n\in X$, $\alpha(n)=0$ and
\item[(b)] for all $\beta\in A$, if $(\forall n\in\ben\setminus X)\big(\beta(n)=\alpha(n)\big)$,
then $\beta \in P$.
\end{itemize}
\end{lemma}

\begin{proof} By Lemma \ref{lembairemeas}(1), (2), and (4),
$P$ has the property of Baire in $A$ and is trivially 
nonmeagre.  Let $f:A\to\{0,1\}$ be the characteristic function of $P$. 
Pick $U$ as guaranteed for $f$ by Theorem \ref{thmMS}.

We claim that there is some $\alpha\in U\cap P$ such that
$\{n\in\ben:\alpha(n)=0\}$ is infinite.  To see this,
for each $l\in\ben$, let $D_l=\{\alpha\in A:(\forall n>l)(\alpha(n)>0)\}$.
Then each $D_l$ is closed with empty interior.  Thus
$S=(A\setminus U)\cup\bigcup_{l=1}^\infty D_l$
is meagre.  Pick $\alpha\in P\setminus S$.  Since
$\alpha\notin \bigcup_{l=1}^\infty D_l$, $Y=\{n\in\ben:\alpha(n)=0\}$ is infinite.
Pick $F$ as guaranteed by Theorem \ref{thmMS} for $\alpha$ and
let $G=\emptyset$.  Pick $X\subseteq Y$ such that
$(\forall\beta\in A)\big((\forall n\in\ben\setminus X)
(\beta(n)=\alpha(n))\Rightarrow f(\beta)=f(\alpha)\big)$.
Since $\alpha\in P$ one has $f(\alpha)=1$ so $X$ is as required.\end{proof}

By Lemmas \ref{Ebig} and \ref{lembaire} we have completed the
proof of Theorem \ref{kHgc} for the case that $Z$ is nonmeagre with the
property of Baire.

The proofs of the following two lemmas are based on the proof
of \cite[Theorem 3]{MS}.

\begin{lemma}\label{MSlem} Let $S$ be a closed subset of 
$A$ and let $\epsilon>0$.  Then there exists $n_0\in\ben$
such that for all $n\geq n_0$, 
$\mu\left(\bigcap_{\eta\in Q_{\{n\}}}(S+\eta)\right)>\mu(S)-\epsilon$.
\end{lemma}

\begin{proof} Pick open $W$ such that $S\subseteq W$ and
$\mu(W)<\mu(S)+\frac{\epsilon}{m}$.  Since $S$ is compact,
pick $r\in\ben$ and basic open sets $U_1,U_2,\ldots,U_r$
such that $S\subseteq \bigcup_{t=1}^r U_t\subseteq W$.
We may then presume that $W=\bigcup_{t=1}^r U_t$.  
For $t\in\nhat{r}$ pick finite $F_t\subseteq\ben$ such that
if $\alpha\in U_t$ and $\delta\in A$ such that $\delta(n)=\alpha(n)$
for all $n\in F_t$, then $\delta\in U_t$. Let
$n_0=\max\bigcup_{t=1}^rF_t+1$.

Let $n\geq n_0$. If $\eta\in Q_{\{n\}}$, then $W+\eta=W$.
Now $$\textstyle\mu(S)+\mu(W\setminus S)=\mu(W)<\mu(S)+\frac{\epsilon}{m}$$
so for $\eta\in Q_{\{n\}}$, 
$\mu\big(W\setminus(S+\eta)\big)=\mu\big((W+\eta)\setminus(S+\eta)\big)
=\mu(W\setminus S)<\frac{\epsilon}{m}$.
Therefore, $\mu\left(W\setminus\bigcap_{\eta\in Q_{\{n\}}}(S+\eta)\right)
\leq \sum_{\eta\in Q_{\{n\}}}\mu\big(W\setminus (S+\eta)\big)<\epsilon$
so $\mu\left(\bigcap_{\eta\in Q_{\{n\}}}(S+\eta)\right)>\mu(W)-\epsilon
\geq \mu(S)-\epsilon$.\end{proof}

\begin{lemma}\label{lemTP} Let $P$ be a measurable subset of $A$ such
that $\mu(P)>0$. Then there exist $T\subseteq P$ and an infinite
subset $Y$ of $\ben$ such that $\mu(T)>0$ and
for all $\alpha\in T$ and all $\beta\in A$, if 
$(\forall n\in\ben\setminus Y)\big(\beta(n)=\alpha(n)\big)$,
then $\beta \in P$.
\end{lemma}

\begin{proof} Let $\epsilon=\frac{\mu(P)}{2}$.  Pick closed
$S_0\subseteq P$ such that $\mu(S_0)>\mu(P)-\epsilon$. By 
Lemma \ref{MSlem} (using $\epsilon'=\mu(S_0)-\mu(P)+\epsilon$)
pick $n_1\in\ben$ such that for all $n\geq n_1$, 
$\mu\left(\bigcap_{\eta\in Q_{\{n\}}}(S_0+\eta)\right)>\mu(P)-\epsilon$.
Let $S_1=\bigcap_{\eta\in Q_{\{n_1\}}}(S_0+\eta)$.  Inductively,
let $r\in\ben$ and assume we have $n_r$ and $S_r$ such that
$\mu(S_r)>\mu(P)-\epsilon$. By Lemma \ref{MSlem}
pick $n_{r+1}>n_r$ such that for all $n\geq n_{r+1}$, 
$\mu\left(\bigcap_{\eta\in Q_{\{n\}}}(S_r+\eta)\right)>\mu(P)-\epsilon$
and let $S_{r+1}=\bigcap_{\eta\in Q_{\{n_{r+1}\}}}(S_r+\eta)$.
Note that $S_{r+1}\subseteq S_r$.

Let $T=\bigcap_{r=0}^\infty S_r$, let $Y=\{n_r:r\in\ben\}$, and
note that $\mu(T)\geq\mu(P)-\epsilon$.  Note that
for any $r\in\ben$, $\eta\in Q_{\{n_r\}}$, and $\alpha\in S_r$, 
$\alpha+\eta\in S_{r-1}$ (since $-\eta\in Q_{\{n_r\}})$.  Therefore
if $\eta\in Q_{\{n_1,n_2,\ldots,n_r\}}$ and $\alpha\in S_r$,
then $\alpha+\eta\in S_0$.

Now let $\alpha\in T$ and $\beta\in A$ and assume that
$(\forall n\in\ben\setminus Y)\big(\beta(n)=\alpha(n)\big)$.
For $r\in\ben$ define $\beta_r\in A$ by, for $n\in\ben$,
$$\beta_r(n)=\cases{\beta(n)&if $n\leq n_r$\cr
\alpha(n)&if $n>n_r$.\cr}$$
Now for each $r$, $\beta_r=\alpha+(\beta_r-\alpha)$, $\alpha\in S_r$,
and $\beta_r-\alpha\in Q_{\{n_1,n_2,\ldots,n_r\}}$, so $\beta_r\in S_0$.
Since $\displaystyle \lim_{r\to\infty}\beta_r=\beta$ and $S_0$ is closed,
$\beta\in S_0\subseteq P$.
\end{proof}

\begin{lemma}\label{lemmeasure} Let $Z$ be a measurable subset
of $(0,1)$  such that $\lambda(Z)>0$ and let 
$P=\psi^{-1}[Z\cap W]$. Then there exist $\alpha\in P$ and an infinite
subset $X$ of $\ben$ such that 
\begin{itemize}
\item[(a)] for all $n\in X$, $\alpha(n)=0$ and
\item[(b)] for all $\beta\in A$, if $(\forall n\in\ben\setminus X)\big(\beta(n)=\alpha(n)\big)$,
then $\beta \in P$.
\end{itemize}
\end{lemma}

\begin{proof}  By Lemma \ref{lembairemeas}(3), (5), and (6)
and the fact that $\lambda(C)=\mu(A\setminus B)=0$, we have
that $P$ is measurable and $\mu(P)>0$.
Pick $T\subseteq P$ and an infinite subset $Y$ of $\ben$ as
guaranteed for $P$ by Lemma \ref{lemTP}.

For each $l\in\ben$, let $D_l=\{\alpha\in A:(\forall n\in Y)(n>l\Rightarrow\alpha(n)>0)\}$.
We claim that $\mu(D_l)=0$.  To see this, enumerate $\{n\in Y:n>l\}$ as
$\langle n_t\rangle_{t=1}^\infty$.  Given $r\in\ben$, let
$H_r=\{\delta\in A:(\forall t\in\nhat{r})(\delta(n_t)>0)\}$.
Then $D_l\subseteq H_r$ and $\mu(H_r)=\left(\frac{m-1}{m}\right)^r$.

Since $(A\setminus B)$ is countable, $\mu(A\setminus B)=0$.
Pick $\alpha\in T\setminus\big((A\setminus B)\cup\bigcup_{l=1}^\infty D_l\big)$.
Let $X=\{n\in Y:\alpha(n)=0\}$. Since $\alpha\in T$ and $X\subseteq Y$, we have that
for all $\beta\in A$, if 
$(\forall n\in\ben\setminus X)\big(\beta(n)=\alpha(n)\big)$,
then $\beta \in P$.
\end{proof}

By Lemmas \ref{Ebig} and \ref{lemmeasure} we have completed the
proof of Theorem \ref{kHgc} for the case that $Z$ is measurable and
$\lambda(Z)>0$ and so Theorem \ref{kHgc} has been established.
Corollary \ref{corkHgc} follows immediately.

\bibliographystyle{plain}

\end{document}